\documentclass[12pt]{article} 
\usepackage{t1enc}
\usepackage[latin1]{inputenc}
\usepackage[english]{babel}
\usepackage{amsmath,amsthm}
\usepackage{amsfonts}
\usepackage{latexsym}
\usepackage[dvips]{graphicx}
\usepackage{graphicx}
\usepackage[natural]{xcolor}
\usepackage{epstopdf}
\theoremstyle{plain}

\newtheorem{theorem}{Theorem}
\numberwithin{theorem}{section}
\newtheorem{lemma}{Lemma}
\numberwithin{lemma}{section}
\newtheorem{corollary}{Corollary}
\numberwithin{corollary}{section}

\newtheorem{conjecture}{Conjecture}
\numberwithin{conjecture}{section}
\newtheorem*{conjecture*}{Conjecture}

\usepackage{graphicx}
\usepackage[natural]{xcolor}
\usepackage{tikz}
\usepackage{tikz-3dplot}
\usetikzlibrary{shapes}
\usetikzlibrary{arrows,decorations.pathmorphing,backgrounds,positioning,fit,petri,automata}
\usetikzlibrary{positioning}

\title{The edge colorings of $K_{5}$-minor free graphs
\thanks{This work is supported by NSFC (11971270, 11901263, 61802158)}}
\author{Jieru Feng$^{a}$, Yuping Gao$^{b}$\thanks{Corresponding author. \emph{E-mail address:} gaoyp@lzu.edu.cn.}, Jianliang Wu$^{a}$
\\
{\small a. School of Mathematics, Shandong University, Jinan 250100, China}\\
{\small b. School of Mathematics and Statistics, Lanzhou University, Lanzhou 730000, China}}
\date{}
\begin{document}
\baselineskip 0.7cm

\maketitle
\begin{abstract}
In 1965, Vizing proved that every planar graph $G$ with maximum degree $\Delta\geq 8$ is edge $\Delta$-colorable. It is also proved that every planar graph $G$ with maximum degree $\Delta=7$ is edge $\Delta$-colorable by Sanders and Zhao, independently by Zhang. In this paper, we extend the above results by showing that every $K_5$-minor free graph with maximum degree $\Delta$ at least seven is edge $\Delta$-colorable.

\end{abstract}

\section{Introduction}

All graphs considered in this paper are finite, undirected and simple. Let $G$ be a graph. We use $V(G),E(G),\Delta(G)$ and $\delta(G)$ (or simply $V,E,\Delta$ and $\delta$) to denote the vertex set, the edge set, the maximum degree and the minimum degree of $G$, respectively. For a vertex $v\in V(G)$, let $N_{G}(v)=\{u\in V(G)\mid uv\in E(G)\}$ be the set of neighbors of $v$. Furthermore, let $N_G(X)=\bigcup_{u\in X}N_G(u)\backslash X$  for a subset $X\subseteq V(G)$. A $k$-\emph{cycle} is a cycle of length $k$. A 3-cycle is also said to be a \emph{triangle}.

An \emph{edge $k$-coloring} of a graph $G$ is an assignment of $k$ colors $1,2,\cdots,k$ to the edges of $G$ such that no two adjacent edges receive the same color. The minimum integer $k$ such that $G$ admits an edge $k$-coloring is called the \emph{chromatic index} of $G$ and is denoted by $\chi'(G)$. For any graph $G$, it is obviously that $\chi'(G)\geq \Delta(G)$. Vizing \cite{Vizing1964} and Gupta \cite{Gupta1967} independently proved that $\chi'(G)\leq \Delta(G)+1$. This leads to a natural classification of graphs into two classes. A graph is said to be \emph{class} 1 if $\chi'(G)=\Delta(G)$ and of \emph{class} 2 if $\chi'(G)=\Delta(G)+1$.

The problem of deciding whether a graph is class 1 or class 2 is NP-hard, see Holyer \cite{Holyer1981}. It is reasonable to consider the problem for some special classes of graphs, such as planar graphs. In \cite{Vizing1968}, Vizing gave some examples of planar graphs with maximum degree at most five which are of class 2. He also proved that any planar graph with maximum degree at least eight is of class 1 \cite{Vizing1965}. Planar graphs with maximum degree seven are of class 1 are proved by Sanders and Zhao \cite{SZ2001}, independently by Zhang \cite{Zhang2000}. It remains an open problem that any planar graph with maximum degree six is of class 1. This problem is affirmative provided some additional conditions, see  \cite{ZW2018}.

By \emph{contracting} an edge $e$ of a graph $G$, we mean that deleting $e$ from $G$ and then identifying its end-vertices and deleting all multiple edges. A graph $H$ is a \emph{minor} of a graph $G$ if $H$ can be obtained from $G$ by deleting edges, deleting vertices and contracting edges. A graph $G$ is called $H$-\emph{minor free} if $G$ has no minor which is isomorphic to $H$. It is well-known that every planar graph contains neither $K_{5}$-minor nor $K_{3,3}$-minor. Therefore, the family of $K_5$-minor free graphs is a generalization of planar graphs. The goal of this paper is to extend the result from planar graphs in \cite{SZ2001,Zhang2000} to $K_5$-minor free graphs. The main result of this paper is as follows.

\begin{theorem}\label{th1}
Let $G$ be a $K_{5}$-minor free graph with maximum degree $\Delta(G)\geq 7$. Then $\chi'(G) =\Delta(G)$.
\end{theorem}

The remainder of this paper is organized as follows. In Section 2, we prove several
properties of $K_5$-minor free graphs; and in Section 3, we prove Theorem \ref{th1} based on the results in Section 2.

\section{Structural properties of $K_5$-minor free graphs}

Before proceeding, we introduce some notation. Let $G$ be a planar graph which is embedded in the plane. Denote by $F(G)$ the face set of $G$. For a face $f\in F(G)$, the \emph{degree} $d(f)$ of $f$ is the number of edges incident with it, where each cut-edge is counted twice. A \emph{$k$-face}, \emph{$k^{-}$-face} and \emph{$k^{+}$-face} (resp. \emph{$k$-vertex}, \emph{$k^{-}$-vertex} and \emph{$k^{+}$-vertex}) is a face (resp. vertex) of degree $k$, at most $k$ and at least $k$, respectively.

\begin{lemma}\label{l1}
Let $G$ be a planar graph of the maximum degree $7$ and $Y$$(1 \leq|Y|\leq3)$ be a subset of nonadjacent vertices of $G$ on the same face $f_{0}$ such that $H= G\setminus Y$ has at least one edge. For any vertex $u\in V(H)$, let $N_Y(u)=\{v\in Y\mid uv\in E(G)\}$. Suppose that
\begin{description}
    \item[$(a)$] $d_{G}(x)\geq3$ for any vertex $x\in V(H)$,
    \item[$(b)$] for any edge $xy\in E(H)$, $x$ is adjacent to at least $(8-d_G(y)-|N_Y(x)|)$ $7$-vertices of $G$ other than $y$ and $N_Y(x)$, and
    \item[$(c)$] for any edge $xy\in E(H)$, if $d_G(x)<7$, $d_G(y)<7$ and $d_G(x)+d_G(y)=9$, then every vertex of $N_H(N_H(\{x,y\}))\backslash\{x,y\}$ is a $7$-vertex of $G$.
\end{description}
Then there is a vertex $x\in V(H)$ satisfying at least one of the following conditions.
\begin{description}
  \item[$(1)$] $x$ is adjacent to two vertices $y, z$ of $H$ such that $d_{G}(z)< 16- d_{G}(x)- d_{G}(y)$ and $xz$ is incident with at least $d_{G}(x)+d_{G}(y)-9$ triangles not containing $y$;
  \item[$(2)$] $x$ is adjacent to four vertices $v, w, y, z$ of $H$ such that $d_{G}(w)\leq 5$, $d_{G}(y)= d_{G}(z)=5$ and $vwx, xyz$ are triangles;
  \item[$(3)$] $x$ is adjacent to four vertices $v, w, y, z$ of $H$ such that $d_{G}(v)< 7$, $d_{G}(w)<7$, $d_{G}(y)= d_{G}(z)= 5$ and $xyz$ is a triangle.
\end{description}
\end{lemma}

\begin{proof}
The proof is carried out by contradiction. Let $G$ be a counterexample to the lemma. By Euler's formula $|V|-|E|+|F|=2$, we have that
 $$\sum\limits_{v\in V(G)}(d(v)-6)+\sum\limits_{f\in F(G)}(2d(f)-6)= -12. $$
 That is
 $$ \sum\limits_{v\in V(G)}(d(v)-6)+\sum\limits_{f\in F(G)\backslash f_0}(2d(f)-6)+(2d(f_0)+6)=0.$$

We define $ch$ to be the {\em initial charge} by letting $ch(x)=d(x)-6$ if $x\in V(G)$, $ch(x)=2d(x)-6$ if $x\in F(G)\backslash f_0$ and $ch(x)=2d(x)+6$ if $x=f_0$. Let $VF(G)=V(G)\cup F(G)$. Thus, we have

$$\sum_{x\in VF(G)}ch(x)=0.$$

Denote by $X=V(H)$.  We define a \emph{hi-vertex} of $G$ to be a $7$-vertex in $X$ or a vertex in $Y$. For a vertex $v\in X$, let $N_X(v)=\{u\in X\mid uv\in E(G)\}$ and $\delta_X(v)=\min\{d_G(u)\mid u\in N_X(v)\}$. Thus $(b)$ is equivalent to that
\begin{description}
    \item[$(d)$]{\it for any edge $xy\in E(H)$, $x$ is adjacent to at least $8- d_G(y)$ hi-vertices other than $y$.}
\end{description}
Now we define the discharging rules as follows.

{\it
\begin{description}
\item[$\bf R1.$] Let $f$ be a face of $G$ and $x$ be a vertex incident with $f$.
  \begin{description}
    \item[$\bf R1.1.$] Suppose that $f=f_0$. If $x\in Y$, then $f_0$ sends $6$ to $x$; Otherwise let $Z$ be the set of vertices adjacent to $x$ and incident with $f_0$. If $|Z\backslash Y|=1$, then $f_0$ sends $1$ to $x$. If $|Z\backslash Y|=2$, then $f_0$ sends $2$ to $x$.
    \item[$\bf R1.2.$] Suppose that $f(\neq f_0)$ is a $4^+$-face. Firstly, $f$ sends $\frac 12$ to each of its incident vertices. Then for each incident hi-vertex $v$ of $f$, $v$ sends $\frac 14$ to each $6^-$-vertex $u\in N_X(v)$ if $uv$ is incident with $f$$($if exists$)$.
  \end{description}
\item[$\bf R2.$] Let $x\in X$.
  \begin{description}
  \item[$\bf R2.1.$] If $x$ is adjacent to a vertex $y\in Y$, $y$ sends $1$ to $x$.
  \item[$\bf R2.2.$] Suppose that $d_G(x)=3$. For each $6$-vertex $y\in N_X(x)$, $y$ sends $1$ to $x$. For each $7$-vertex $y\in N_X(x)$, if $xy$ is incident with two $3$-faces, then $y$ sends $1$ to $x$, otherwise $y$ sends $\frac{1}{2}$ to $x$.
  \item[$\bf R2.3.$] Suppose that $d_G(x)=4$. If $x$ is adjacent to a $5$-vertex $z\in X$, then for each $7$-vertex $y\in N_X(x)$, $y$ sends $\frac{2}{3}$ to $x$, otherwise for each $6$-vertex $y\in N_X(x)$, $y$ sends $\frac{2}{5}$ to $x$ and for each $7$-vertex $y\in N_X(x)$, if $xy$ is incident with two $3$-faces, then $y$ sends $\frac{3}{5}$ to $x$, otherwise $y$ sends $\frac{1}{5}$ to $x$.
  \item[$\bf R2.4.$] Suppose that $d_G(x)=5$. If $x$ is adjacent to a $4$-vertex $z\in X$, then for each $7$-vertex $y\in N_X(x)$, $y$ sends
  $\frac{1}{3}$ to $x$, otherwise for each $6^+$-vertex $y \in N_X(x)$ such that $xy$ is incident with two $3$-faces,
  if $xy$ is incident with exactly one $(5, 5, 7)$-face, then $y$ sends $\frac{2}{5}$ to $x$, otherwise $y$ sends $\frac{1}{5}$ to $x$.
  \item[$\bf R2.5.$] Suppose that $d_G(x)=6$. If $x$ is adjacent to a $3$-vertex $z\in X$, then for each $7$-vertex $y\in N_X(x)\backslash N_X(z)$, $y$ sends $\frac13$ to $x$, 
       otherwise for each $7$-vertex $y\in N_X(x)$ such that $xy$ is incident with two $3$-faces,  if $xy$ is not incident with two $(6, 7, 7)$-faces, then $y$ send $\frac{1}{5}$ to $x$.
  \end{description}
\end{description}}

Let $ch'(x)$ be the new charge according to the above discharging rules for each $x\in VF(G)$.  Since our rules only move charges around and do not affect the sum, we also have that

$$\sum_{x\in VF(G)}ch'(x)=\sum_{x\in VF(G)}ch(x)= 0.$$

In the following,  we shall show that $ch'(x)\geq 0$ for each $x\in VF(G)$ and $\sum\limits_{x\in VF(G)}ch'(x)>0$ to obtain a contradiction.

Let $f$ be a face of $G$. Suppose that $f=f_0$. R1.1 is equivalent to that for any $y\in Y$, there is a vertex incident with $f_0$ receives nothing from $f_0$. So $ch'(f_0) =ch(f_0)-|Y|\times6-(d(f_0)-2|Y|)\times2\geq 0$. Suppose that $f\not=f_0$. If $d(f)=3$, then $ch'(f)=ch(f)=2d(f)-6=0$; Otherwise $d(f)\geq 4$ and $ch'(f)\geq ch(f)-\frac{1}{2}\times d(f)\geq0$ by R1.2.

If $v\in Y$, then $ch'(v)\leq ch(v)+6-d(v)=0$ by R1.1 and R2.1. So in the following, assume that $v\in X$. We consider the following two cases.

\vspace{3mm}
\noindent
\textbf{Case 1.} $v$ is not incident with $f_0$.

\vspace{3mm}
\noindent
\textbf{Subcase 1.1.} $d_G(v)=3$.

\vspace{3mm}
Then $\delta_X(v)\geq 6$ and at most one vertex in $N_X(v)$ is a $6$-vertex by $(b)$. If $v$ is incident with three $3$-faces, then $ch'(v)\geq ch(v)+1\times3\geq0$ by R2.1 and R2.2. Suppose that $v$ is incident with two 3-faces and a $4^+$-face $f$. If there is a $6$-vertex $u\in N_X(v)$ and $uv$ is incident with $f$, then $v$ receives $\frac 34$ from $f$ by R1.2, $1$ from $u$, at least $\frac 12$ from its adjacent hi-vertex incident with $f$, 1 from another adjacent hi-vertex which is not incident with $f$ by R2.1 and R2.2, and it follows that $ch'(v)\geq ch(v)+ 1+ 1+ \frac 34+ \frac 12> 0$; Otherwise $v$ receives 1 from $f$  by R1.2, at least $\frac 12$ from each of its adjacent hi-vertices incident with $f$, 1 from another adjacent vertex which is not incident with $f$  by R2.1 and R2.2, and it follows that $ch'(v)\geq ch(v)+ 1+ 1+ \frac 12\times2= 0$. If $v$ is incident with at least two $4^+$-faces, then $v$ receives totally at least $\frac 34\times2$ from its incident faces by R1.2, $v$ receives totally at least $\frac 12\times 3$ from its adjacent vertices by R2.1 and R2.2, and it follows that $ch'(v)\geq ch(v)+ \frac 34\times2+\frac 32 = 0$.

\vspace{3mm}
\noindent
\textbf{Subcase 1.2.} $d_G(v)=4$.

\vspace{3mm}
Then $\delta_X(v)\geq 5$ by $(b)$. If there is a $5$-vertex in $N_X(v)$, then $v$ is adjacent to three hi-vertices of $G$ by $(d)$, and then $ch'(v)\geq ch(v)+ \frac{2}{3} \times (3-|N_Y(v)|)+|N_Y(v)|\geq0$ by R2.1 and R2.3; Otherwise $\delta_X(v)\geq 6$ and $v$ is adjacent to at least two hi-vertices of $G$. For each adjacent hi-vertex $u$ of $v$, if $uv$ is incident with two 3-faces, then $u$ sends at least $\frac 35$ to $v$ by R2.1 and R2.3; Otherwise $u$ sends at least $\frac 15$ to $v$ by R2.1 and R2.3. Let $f$ be some $4^+$-face incident with $uv$. By R1.2, $u$ also sends $\frac 14$ to $v$ and $f$ sends $\frac 12$ to $v$. We can split a half of the $\frac 12$ to $u$, that is, it is thought that $u$ sends $\frac 14+\frac 14$ to $v$ by R1.2. Note that $\frac 35<\frac 15+\frac 12$. Since each $6$-vertex in $N_X(v)$ sends $\frac 25$ to $v$ by R2.3, $ch'(v)\geq ch(v)+ \frac{3}{5}\times 2+ \frac{2}{5} \times 2 \geq0$.

\vspace{3mm}
\noindent
\textbf{Subcase 1.3.} $d_G(v)=5$.

\vspace{3mm}
Then $\delta_X(v)\geq 4$ by $(b)$. If $N_Y(v)\neq\emptyset$, then $ch'(v)\geq ch(v)+1= 0$ by R2.1. So assume $N_Y(v)=\emptyset$. If there is a $4$-vertex in $N_X(v)$, then $v$ is adjacent to four 7-vertices by $(d)$ and it follows from R2.4 that $ch'(v)\geq ch(v)+ \frac{1}{3} \times4 >0$. If $\delta_X(v)\geq 6$, then $ch'(v)\geq ch(v)+ \frac{1}{5}\times 5 = 0$ by R1.2 and R2.4. Suppose that $\delta_X(v)=5$. Then $v$ is adjacent to at least three 7-vertices by $(d)$. If $v$ is incident with a $4^{+}$-face, $ch'(v)\geq ch(v)+\frac{1}{5}\times3+\frac{1}{2}>0$ by R1.2 and R2.4; Otherwise, $v$ is incident with five $3$-faces, in this case, there is only one $5$-vertex $u$ in $N_X(v)$ by (1) and $uv$ is incident with a $(5,5,7)$-face, so $ch'(v)\geq ch(v)+ \frac{1}{5} \times 3+ \frac{2}{5}= 0$ by R2.4.

\vspace{3mm}
\noindent
\textbf{Subcase 1.4.} $d_G(v)=6$.

\vspace{3mm}
Suppose that there is a $3$-vertex $u\in N_X(v)$. Then $v$ is adjacent to at least five hi-vertices of $G$ by $(d)$ and at least three of these hi-vertices are not adjacent to $u$. If $N_Y(v)= \emptyset$, then $ch'(v)\geq ch(v)+\frac 13\times (5-2)-1= 0$ by R2.2 and R2.5; Otherwise $ch'(v)\geq ch(v)+\frac 13\times (5-2-|N_Y(v)|)+ |N_Y(v)|-1> 0$ by R2.1.  If $\delta_X(v)\geq 6$, then $v$ sends nothing out and it follows that $ch'(v)\geq ch(v)=0$. Now  we consider the case $4\leq \delta_X(v)\leq 5$.

Let $N_G(v)=\{v_1, v_2,\cdots, v_6\}$ such that $d_G(v_1)=\delta_X(v)$, $vv_{i}$ and $vv_{i+1}$ are incident with the face $f_i$ for any $1\leq i \leq5$ and $vv_{6}$ and $vv_{1}$ are incident with the face $f_6$. Suppose that $d_G(v_1)=4$. Then $v$ is adjacent to at least four hi-vertices by $(d)$ and $v$ sends at most $\frac{2}{5}$ to each adjacent $5^{-}$-vertex by R2.3 and R2.4. If $N_Y(v)\neq\emptyset$, then $ch'(v)\geq ch(v)-\frac 25\times2+1>0$ by R2.1; Otherwise we have $N_Y(v)=\emptyset$. At this time, if there exists a $5^-$-vertex $v_j\in N_X(v)\backslash \{v_1\}$ $(2\leq j\leq6)$, then $vv_j$ is incident with two $4^+$-faces by (d) and (1), and it follows that $v$ sends at most $\frac 25\times 2$ to $v_1$ and $v_j$, but $v$ receives at least $\frac 12\times 2$ by R1.2 from its incident faces, so $ch'(v)\geq ch(v)-\frac 45+1>0$; Otherwise $v$ just sends $\frac 25$ to $v_1$. If some $f_i(1\leq i\leq 6)$ is a $4^+$-face, then $ch'(v)\geq ch'(v)-\frac 25+\frac 12>0$; Otherwise $f_1, f_2,...,f_6$ are 3-faces. Let $j=\min\{i|\, d_G(v_i)=7, 2\leq i\leq 6\}$ and $k=\max\{i|\, d_G(v_i)=7, 2\leq i\leq 6\}$. Then $f_{j-1}$ and $f_k$ are not $(6,7,7)$-faces. Since $v$ is adjacent to at least four hi-vertices, $j<k-1$. By R2.5, $v_{j}$ sends $\frac 15$ to $v$, $v_k$ sends $\frac 15$ to $v$ and it follows that $ch'(v)\geq ch'(v)-\frac 25+\frac 15\times 2=0$.

Suppose that $d_G(v_{1})=5$. Let $N_5=\{d_G(u)=5|\, u\in N_X(v)\}$. Then $1\leq |N_5|\leq 3$ by $(d)$ and $v$ sends at most $\frac{1}{5}$ to each vertex of $N_5$ by R2.4. If $N_Y(v)\neq\emptyset$, then $ch'(v)\geq ch(v)-\frac 15\times3+1>0$ by R2.1; Otherwise $N_Y(v)=\emptyset$. If $|N_5|=1$, then $f_k$ is not a $(6,7,7)$-face, where $k=\max\{i|\, d_G(v_i)=7, 2\leq i\leq 6\}$, and it follows that $ch'(v)\geq ch'(v)-\frac 15+\frac 15=0$ by R1.2 and R2.5. Suppose $|N_5|=2$. If some $f_i(1\leq i\leq 6)$ is a $4^+$-face, then $ch'(v)\geq ch'(v)-\frac 15\times 2+\frac 12>0$ by R1.2; Otherwise $f_1, f_2,...,f_6$ are 3-faces. Let $j=\min\{i|\, d_G(v_i)=7, 2\leq i\leq 6\}$ and $k=\max\{i|\, d_G(v_i)=7, 2\leq i\leq 6\}$. Then $v_{j}$ and $v_{k}$ send $\frac 15\times 2$ to $v$ and it follows that $ch'(v)\geq ch'(v)-(\frac 15-\frac 15)\times 2=0$. Suppose $|N_5|=3$. If there is a $7$-vertex in $N_X(v)$ sends no charge to $v$, then $f_1$ is a $4^+$-face by (3), and $v$ receives at least $\frac 15$ from each of another two 7-vertices by R1.2 and R2.5, it follows that $ch'(v)\geq ch(v)-\frac15\times3+\frac15\times2+\frac12>0$; Otherwise there are three 7-vertices in $N_X(v)$ each of which sends at least $\frac 15$ to $v$ by R1.2 and R2.5, it follows that $ch'(v)\geq ch(v)-(\frac 15+\frac 15)\times 3=0$.

\vspace{3mm}
\noindent
\textbf{Subcase 1.5.} $d_G(v)=7$.

\vspace{3mm}
If $\delta_X(v)=7$, then $ch'(v)\geq ch(v)>0$. Suppose that $\delta_X(v)=6$. If $v$ sends $\frac{1}{3}$ to a $6$-vertex $u$ of $G$, then there is a $3$-vertex $w\in N_X(u)$ and it follows from $(c)$ and $(d)$ that $v$ is adjacent to six hi-vertices and then $ch'(v) \geq ch(v)- \frac{1}{3}>0$ by R2.5; Otherwise each $6$-vertex in $N_X(v)$ receives at most $\frac{1}{5}$ from $v$ and $ch'(v)\geq ch(v)-\frac{1}{5}\times5=0$.

Suppose that $\delta_X(v)=5$. If $v$ sends $\frac{1}{3}$ to a $5$-vertex $u$, then there is a $4$-vertex $w\in N_X(u)$ and it follows from $(c)$ that every neighbor of $v$ except $u$ and $w$ is a hi-vertex, and then $ch'(v) \geq ch(v)- \frac{1}{3}- \frac{2}{3}=0$ by R2.3 and R2.4. If $N_Y(v)\neq\emptyset$, then $v$ sends at most $\frac 25$ to  each $6^-$-vertices in $N_X(v)$ by R2.4 and R2.5, and then $ch'(v)\geq ch(v)+1-\frac 25\times4>0$. So $N_Y(v)=\emptyset$. Suppose that $v$ sends $\frac{2}{5}$ to a $5$-vertex of $G$. Then $v$ is incident with a $(5,5,7)$-face $(u, w, v)$ by R2.4, and it follows from (3) that there are at most one $6^{-}$-vertex in $N_X(v)\backslash\{u,w\}$.  If all vertices in $N_X(v) \backslash \{u, w\}$ are 7-vertices, then  $ch'(v) \geq ch(v)- \frac{2}{5} \times 2>0$; Otherwise $N_X(v)\backslash \{u, w\}$ has a $6^-$-vertex which receives at most $\frac{1}{5}$ from $v$, so $ch'(v) \geq ch(v)- \frac{2}{5} \times 2- \frac{1}{5} \geq0$.  The final case is that each $6^{-}$-vertex in $N_X(v)$ receives at most $\frac{1}{5}$ from $v$ and $ch'(v)\geq ch(v)-\frac{1}{5}\times4>0$.

Suppose that $\delta_X(v)=4$. Then $v$ is adjacent to four hi-vertices by $(d)$. If $v$ sends $\frac{2}{3}$ to a $4$-vertex $u$, then there is a $5$-vertex $w\in N_X(u)$ and it follows from (c) that every neighbor of $v$ except $u$ and $w$ is a hi-vertex, and then $ch'(v) \geq ch(v)- \frac{1}{3}- \frac{2}{3}=0$ by R2.3 and R2.4.  If $N_Y(v)\neq\emptyset$, then $v$ sends at most $\frac 35$ to  each $6^-$-vertices in $N_X(v)$ by R2.3, R2.4 and R2.5, so $ch'(v)\geq ch(v)+1-\frac 35\times3>0$; Otherwise $N_Y(v)=\emptyset$. Suppose that $v$ sends $\frac{3}{5}$ to a $4$-vertex $u$. Then there is no 4-vertex in $N_X(v)\backslash u$  by $(b)$ and $(1)$, no vertex receives $\frac{2}{5}$ from $v$ by R2.4 and $(2)$, and it follows that $ch'(v)\geq ch(v)-\frac{3}{5}-\frac{1}{5}\times2\geq0$. If $v$ sends at most $\frac{1}{5}$ to each $4$-neighbor of $v$, then $ch'(v)\geq ch(v)- \frac{1}{5}- \frac{2}{5}\times2\geq0$ by R2.4 and R2.5.

Suppose that $\delta_X(v)=3$. Then $v$ is adjacent to five hi-vertices by $(d)$. If there two $3$-vertices $u_1,u_2$ in $N_X(v)$, then $vu_1$ and $vu_2$ are incident with no 3-faces by $(1)$ and it follows that $v$ sends at most $\frac{1}{2}\times2$ to $u_1$ and $u_2$ by R2.2; Otherwise, let $u\in N_X(v)$ be the unique 3-vertex. If $v$ sends $1$ to it from R2.2, then $v$ sends no other charge (since by (1), there is no 4-vertex in $N_X(v)$ and if $N_X(v)$ contains a 5-vertex $u$ then $uv$ is incident with at least one $4^+$-face); Otherwise $v$ sends at most $\frac{1}{2}$ to each of the two neighbors of $v$ of degree at most six. In either case, $v$ sends out at most 1. Thus, if $N_Y(v)=\emptyset$, then $ch'(v)\geq ch(v)-1=0$; Otherwise $ch'(v)\geq ch(v)-1 +|N_Y(v)|>0$.

\vspace{3mm}
\noindent
\textbf{Case 2.} $v$ is incident with $f_0$.

\vspace{2mm}
Let $Z$ be the set of vertices adjacent to $v$ and incident with $f_0$. Since $d_G(v)\geq 3$, $|Z|\geq 2$. If $|Z\cap Y|\geq 2$, then $v$ receives at least 2 totally from $N_Y(v)$. If $|Z\backslash Y|=1$, then $f_0$ sends $1$ to $v$ and $N_Y(v)$ sends 1 to $v$. If $|Z\backslash Y|=2$, then $f_0$ sends $2$ to $x$. So $v$ receives at least 2 totally from $f_0$ and $N_Y(v)$.

\vspace{3mm}
\noindent
\textbf{Subcase 2.1.} $d_G(v)=3$.

\vspace{2mm}
Then $\delta_X(v)\geq 6$ and there is at most one $6$-vertex in $N_X(v)$ by $(b)$. If $|N_Y(v)|=3$, then $v$ is incident with a $4^+$-face and it follows that $ch'(v)\geq ch(v)+3\times 1 +\frac 12> 0$ by R1.2 and R2.1. Suppose that $|N_Y(v)|=2$. Let $\{u\}=N_X(v)$. If $uv$ is incident with $f_0$, then $v$ receives at least 1 from $f_0$ by R1.1, at least $\frac 12$ from $u$ by R2.2  and it follows that $ch'(v)\geq ch(v)+2 +1 +\frac 12> 0$; Otherwise, if $uv$ is incident with two 3-faces, then $v$ receives 1 from $u$ by R2.2 and it follow that $ch'(v)\geq ch(v)+2 +1= 0$; Otherwise, $v$ receives at least $\frac 12$ from $u$ by R2.2 and at least $\frac 12+\frac 14$ from a $4^+$-face incident with $uv$ by R1.2, so $ch'(v)\geq ch(v)+2 +\frac 14 +1> 0$. Suppose that $|N_Y(v)|=1$. Let $\{u\}=N_Y(v)$, $\{x,y\}=N_X(v)$ and $vx$ be incident with $f_0$. If $vy$ is also incident with $f_0$, then $v$ receives 2 from $f_0$, 1 from $u$, $\frac 12$ from $x$ and it follows that $ch'(v)\geq ch(v)+2 +1+ \frac 12> 0$; Otherwise $v$ receives 1 from $f_0$ by R1.1, receives 1 from $u$ and receive at least $\frac 12$ from $x$. If $vy$ is incident with two 3-face, then $v$ receives $1$ from $y$ and it follows that $ch'(v)\geq ch(v)+ 1\times 3+ \frac 12> 0$; Otherwise $v$ receives $\frac 12$ from $y$, $\frac 12$ from a $4^+$face incident with $vy$ and it follows that $ch'(v)\geq ch(v)+ \frac 12\times 3+ 1\times 2> 0$. Suppose that $|N_Y(v)|=0$. Let $u\in N_X(v)$ such that $uv$ is not incident with $f_0$. Then $v$ receives 2 from $f_0$ by R1.1, receives totally at least $2$ from its neighbors and faces incident with $uv$ by the similar argument as above, so $ch'(v)\geq ch(v)+ 2+ 2> 0$.

\vspace{3mm}
\noindent
\textbf{Subcase 2.2.} $d_G(v)=4$.

\vspace{2mm}
If $|N_Y(v)|=3$, then $ch'(v)\geq ch(v)+3> 0$; Otherwise $|N_X(v)|\geq 2$. If there is a $5$-vertex in $N_X(v)$, then $v$ is adjacent to at least one 7-vertex of $G$ by $(b)$, and it follows from R2.3 that $ch'(v)\geq ch(v)+ 2+ \frac{2}{3}>0$; Otherwise $v$ receives at least $\frac 15$ from $N_X(v)$ by R1.2 and R2.3, then $ch'(v)\geq ch(v)+ 2+ \frac{1}{5}>0$.

\vspace{3mm}
\noindent
\textbf{Subcase 2.3.} $5\leq d_G(v)\leq 7$.

\vspace{2mm}
If $d_G(v)=5$, then $ch'(v)\geq ch(v)+ 2>0$; Otherwise, by the similar argument as Subcase 1.4-1.5, $v$ sends out less than $2$, so $ch'(v)> ch(v)+ 2- 2>0$.

\vspace{3mm}
Till now, we have checked that $ch'(x)\geq 0$ for any element $x\in VF(G)$. Now we begin to find a vertex or a face $x\in VF(G)$ such that $ch(x) > 0$. If $|Y|\leq2$, then $ch'(f_0)>0$. So we assume $|Y|=3$. Let $v\in X$ be a vertex incident with $f_0$. According to Case 2, $ch'(v)=0$ if and only if $d_G(v)=3$, $|N_Y(v)|=2$ and the edge not incident with $f_0$ is incident with two 3-faces. Let $\{u\}=N_X(v)$. Then $d_G(u)\geq 6$ by $(b)$ and it follows from  Subcase 1.4,1.5 and 2.3 that $ch'(u)>0$.

Hence we complete the proof of the lemma.
\end{proof}

Let $G$ be a connected graph, $T$ be a tree, and $ \mathcal{F}=\{V_t\subseteq V(G):t\in V(T)\}$ be a family of subsets of $V(G)$. The ordered set $(T, \mathcal{F})$ is called a $tree$-$decomposition$ of $G$ if it satisfies the following conditions:

(T1) $V(G)=\cup_{t\in V(T)}V_{t}$;

(T2) for any edge $e\in E(G)$, there exists a vertex $t\in V(T)$ such that the two end vertices of $e$ are included in $V_t$;

(T3) if $t_1,t_2,t_3\in V(T)$ and $t_2$ is on the $(t_1,t_3)$-path of $T$, then $V_{t_1}\cap V_{t_3}\subset V_{t_2}$.

For any adjacent vertices $s$ and $t$ in $T$, $V_s\cap V_t$ form a vertex cut of $G$, called as a $separate$ $set$ of the tree decomposition. The graph $G_t=G[V_t]$ for any $t\in V(T)$ is called a $part$ of the tree-decomposition.  If the  induced subgraph of any separate set of the tree-decomposition is a complete graph, the tree decomposition is called  $simple$. If any separate set of a simple tree-decomposition has at most $k$ vertices, the tree decomposition is called $k$-$simple$.

\begin{lemma}{\rm \cite{W1937}}\label{K5}
Let $G$ be an edge-maximal graph without a $K_5$ minor. If $|V(G)|\geq 4$, then $G$ has a $3$-simple tree-decomposition $(T, \mathcal{F})$ such that each part is  a planar triangulation or the Wagner graph $W$$($see Figure $\ref{f})$.
\end{lemma}

\begin{figure}[htbp]
\begin{center}
\includegraphics[scale=0.6]{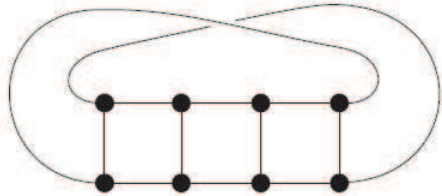} \\
\caption{The Wagner graph $W$}
\label{f}
\end{center}
\end{figure}

The lemma implies that  every $K_5$-minor free graph has a tree-decomposition $(T, \mathcal{F})$ such that each part is a planar graph or the Wagner graph and each separate set has size at most $3$.

\begin{lemma}\label{k5}
Let $G$ be a $K_5$-minor free graph of the maximum degree $7$. Suppose that
\begin{description}
  \item[$(a)$] $\delta(G)\geq 3$,
  \item[$(b)$] for any edge $xy\in E(G)$, $x$ is adjacent to at least $(8-d_G(y))$ $7$-vertices of $G$ other than $y$, and
  \item[$(c)$] for any edge $xy\in E(G)$, if $d_G(x)<7$, $d_G(y)<7$ and $d_G(x)+d_G(y)=9$, then every vertex of $N_G(N_G(\{x,y\}))\backslash\{x,y\}$ is a $7$-vertex of $G$.
\end{description}

Then $G$ contains a vertex $x$ satisfying one of the following conditions:
\begin{description}
  \item[$(1)$] $x$ is adjacent to two vertices $y, z$ such that $d_G(z)< 16- d_G(x)- d_G(y)$ and $xz$ is incident with at least $d_G(x)+d_G(y)-9$ triangles not containing $y$;
  \item[$(2)$] $x$ is adjacent to four vertices $v, w, y, z$ such that $d_G(w)\leq 5$, $d_G(z)=5$, $d_G(y)=5$, and $vwx$ and $xyz$ are triangles;
  \item[$(3)$] $x$ is adjacent to four vertices $v, w, y, z$ such that $d_G(v)< 7$, $d_G(w)< 7$, $d_G(y)= d_G(z)= 5$ and $xyz$ is a triangle.
\end{description}
\end{lemma}

\begin{proof}

Suppose to the contrary that $G$ is a counterexample to the lemma such that $|V(G)|$ is as small as possible. Let  $(T, \mathcal{F})$  be a tree-decomposition of $G$ such that each part is a planar graph or the Wagner graph, each separate set is of size at most $3$ and $|V(T)|$ is as small as possible. Suppose that $|V(T)|=1$. Then $G$ must be a planar graph and $|V(G)|\geq 5$(since the Wagner graph is of the maximum degree $3$, it is contradicted to $(b)$).  Let $v\in V(G)$ and $Y=\{v\}$. Then $G$ satisfies the conditions $(a)$-$(c)$ of Lemma \ref{l1}, it follows that $G$ satisfies the lemma, a contradiction. So $|V(T)|\geq 2$.

Let $v_1v_2...v_t(t\geq 2)$ be a longest path of $T$. Then $v_1$ is a leaf of $T$. By $(b)$, $G_{v_1}$ is a planar graph. Let $S_{12}=V_{v_1}\cap V_{v_2}$, $G'_1=G_{v_1}\backslash S_{12}$ and $G^*_1=G'_1\cup\{xy|\, x\in V(G'_1), y\in S_{12}$ and $xy\in E(G)\}$. Without loss of generality, we  assume that $G'_1$ is connected( for otherwise we can consider a connected component of $G'_1$). If $|V(G'_1)|\geq 2$, then it follows from Lemma \ref{l1} that $G$ contains a vertex satisfying the lemma, a contradiction. If $|V(G'_1)|=1$ and $|S_{12}|\leq 2$, then $\delta(G)\leq 2$, a contradiction to $(a)$. So $|V(G'_1)|=1$ and $|S_{12}|=3$, that is, $G^*_1$ is a star of order 4. Let $V(G^*_1)=\{u, u_{1}, u_{2}, u_{3}\}$ such that $\{u\}=V(G'_1)$ and $S_{12}=\{u_{1}, u_{2}, u_{3}\}$.  Then $d_G(u)=3$.

Since $|S_{12}|=3$, $G_{v_2}$ is a planar graph. Let $K=G_{v_2}\cup \{xy|\, x,y\in S_{12}\; \text{and}\; xy\notin G_{v_2}\}$. Then $K$ is also a planar graph.  We embed $K$ into the plane such that $S_{23}=V_{v_2}\cap V_{v_3}$(if $t=2$, then we choose any vertex of $V_{v_2}$ as $S_{23}$) are located on the unbounded face $f_0$. By the minimality of $T$, The cycle $C=u_{1}u_{2}u_{3}u_{1}$ of $K$ must be a separated triangle and by the similar arguments as above, the inner part of $C$ is equivalent to a leaf of $T$ and is also a star of order $4$. Let $w$ be the inner vertex of $C$. Then $d_G(w)=3$ and $N_G(w)= N_{K}(w)= N_G(u)=\{u_{1}, u_{2}, u_{3}\}$.

By $(1)$, we have $u_{i}u_{j}\notin E(G)$ for any $1\leq i<j\leq 3$, that is, $\{u_{1}, u_{2}, u_{3}\}$ is an independent set of $G$. By $(b)$, we have that $d_G(u_{i})=7$ and all vertices of $N_G(u_i)\backslash \{u,w\}$ are $7$-vertices for any $1\leq i\leq 3$. Let $G'=G\backslash \{u,w\}+\{u_1u_2, u_2u_3, u_3u_1\}$($G'$ is obtained from $G$ by contracting edge $uu_1$ and $wu_2$). So $G'$ is also a $K_5$-minor free graph and is a counterexample. Since $|V(G')|<|V(G)|$, it is a contradiction. We complete the proof of the lemma.
\end{proof}

\section{The proof of Theorem\ref{th1}}

In investigating graph edge coloring problems, critical graphs always play an important role. This is due to the fact that problems for graphs in general may often be reduced to problems for critical graphs whose structure is more restricted. A connected graph $G$ is \emph{critical} if it is class 2, and $\chi'(G-e)<\chi'(G)$ for any edge $e\in E(G)$. A critical graph with the maximum degree $\Delta$ is called a \emph{$\Delta$-critical graph}. It is clear that every critical graph is 2-connected. Before the proof of our main result, we give some structure lemmas of critical graphs as follows.

\begin{lemma}\label{Vizing}
$(${\rm Vizing's Adjacency Lemma \cite{Vizing1965}}$)$ Let $G$ be a
$\Delta$-critical graph, and let $u$ and $v$ be adjacent vertices of $G$ with $d(v)=k$.

$(a)$ If $k<\Delta$, then $u$ is adjacent to at least $\Delta-k+1$ vertices of degree $\Delta$$;$

$(b)$ If $k=\Delta$, then $u$ is adjacent to at least two vertices of degree $\Delta$.
\end{lemma}

From the above Lemma, it is easy to get the following corollary.

\begin{corollary} \label{cor}
Let $G$ be a $\Delta$-critical graph. Then

$(a)$ $\delta(G)\geq2$,

$(b)$ every vertex is adjacent to at most one $2$-vertex and at least two $\Delta$-vertices,

$(c)$ for any edge $uv\in E(G)$, $d_G(u)+d_G(v)\geq \Delta+2$, and

$(d)$ if $uv\in E(G)$ and $d(u)+d(v)=\Delta+2$, then every vertex of $N(\{u,v\})\setminus\{u,v\}$ is a $\Delta$-vertex.
\end{corollary}

\begin{lemma}{\rm \cite{Zhang2000}} \label{zhang}
Suppose that $G$ is a $\Delta$-critical graph, $uv\in E(G)$ and $d(u)+d(v)=\Delta+2$. Then

$(a)$ every vertex of $N_G(N_G(\{u,v\}))\setminus\{u,v\}$ is of degree at least $\Delta-1$$;$

$(b)$ if $d(u),d(v)<\Delta$, then every vertex of $N_G(N_G(\{u,v\}))\setminus\{u,v\}$ is a $\Delta$-vertex.
\end{lemma}

\begin{lemma}{\rm \cite{SZ2001}} \label{zhao2}
No $\Delta$-critical graph has distinct vertices $x, y, z$ such that $x$ is adjacent to $y$ and $z$, $d(z)<2\Delta-d(x)-d(y)+2$, and $xz$ is in at least $d(x)+d(y)-\Delta-2$ triangles not containing $y$.
\end{lemma}

\begin{lemma}{\rm \cite{SZ2001}} \label{zhao3}
No $\Delta$-critical graph has distinct vertices $v, w, x, y, z$ such that $d(w)\leq \Delta-2$, $d(x)+d(y)\leq \Delta+3$, $d(x)\geq5$, $d(y)\geq5$, and $vwz$ and $xyz$ are triangles.
\end{lemma}

\begin{lemma}{\rm \cite{SZ2001}} \label{zhao4}
No $\Delta$-critical graph has distinct vertices $v, w, x, y, z$ such that $d(v)\leq \Delta-1$, $d(w)\leq\Delta-1$, $d(x)+d(y)\leq \Delta+3$, $d(x)\geq4$, $d(y)\geq4$, $xyz$ is a triangle, and $z$ is adjacent to $v$ and $w$.
\end{lemma}

\begin{lemma}{\rm \cite{MW2002}} \label{Miao2002}
If $G$ is a $\Delta$-critical graph with $n$ vertices, where $\Delta\geq 8$. Then $|E(G)|\geq 3(n+\Delta-8)$.
\end{lemma}

\begin{proof}[Proof of Theorem $\ref{th1}$] Suppose, to the contrary, that $H$ is a $\Delta$-critical $K_5$-minor free graph with $\Delta\geq 7$. In \cite{Mader1998}, Mader proved that each $K_5$-minor graph $G$ with $n$ vertices has at most $3n-6$ edges. Therefore, if $\Delta\geq 8$, it is a contradiction to Lemma \ref{Miao2002}. Assume that $\Delta=7$ in what follows.

By Corollary \ref{cor}, we have that $\delta(H)\geq 2$, every vertex of $H$ is adjacent to at most one $2$-vertex and all neighbors of any $2$-vertex in $H$ are 7-vertices. We construct a new graph $H'$ from $H$ by contracting all 2-vertices and deleting all contracted multiple edges. Then $H'$ is also a $K_5$-minor free graph and we have
\begin{description}
  \item[$(a)$] $\delta(H')\geq 3$ and $d_{H'}(v)\geq d_H(v)-1$ for any $v\in V(H')$.
\end{description}
By Corollary \ref{cor}$(d)$ and Lemma \ref{zhang}, if a $7$-vertex $v$ of $H$ is adjacent to $2$-vertex, then all neighbors of $v$ are still $7$-vertices in $H'$. So we have
\begin{description}
  \item[$(b)$] \textit{for any $xy\in E(H')$, $x$ is adjacent to at least $(8-d_{H'}(y))$ $7$-vertices of $H'$ other than $y$}.
  \item[$(c)$] \textit{for any edge $xy\in E(H')$, if $d_{H'}(x)<7$, $d_{H'}(y)<7$ and $d_{H'}(x)+d_{H'}(y)=9$, then every vertex of $N_{H'}(N_{H'}(\{x,y\}))\backslash\{x,y\}$ is a $7$-vertex of $H'$.}
\end{description}
By Lemma \ref{k5}, $H'$ contains a vertex $x$ satisfying one of the following conditions:
\begin{description}
  \item[$(d)$] \textit{$x$ is adjacent to two vertices $y, z$ such that $d_{H'}(z)< 16- d_{H'}(x)- d_{H'}(y)$ and $xz$ is incident with at least $d_{H'}(x)+d_{H'}(y)-9$ triangles not containing $y$};
  \item[$(e)$] \textit{$x$ is adjacent to four vertices $v, w, y, z$ such that $d_{H'}(w)\leq 5$, $d_{H'}(z)=5$, $d_{H'}(y)=5$, and $vwx$ and $xyz$ are triangles};
  \item[$(f)$] \textit{$x$ is adjacent to four vertices $v, w, y, z$ such that $d_{H'}(v)< 7$, $d_{H'}(w)< 7$, $d_{H'}(y)+d_{H'}(z)\leq 10$, $d_{H'}(y)\geq4$, $d_{H'}(z)\geq4$ and $xyz$ is a triangle}.
\end{description}
However, such vertex $x$ does not exist according to Lemma \ref{zhao2}-\ref{zhao4}, a contradiction. We complete the proof of Theorem \ref{th1}.
\end{proof}

\section{Concluding remarks}

In this paper, we showed that every $K_5$-minor free graph with maximum degree $\Delta\geq 7$ is class 1. The idea of the proof of the main theorem can be extended into other colorings of $K_5$-minor free graphs. It is natural to investigate other graph coloring problems for $K_5$-minor free graphs. Furthermore, for edge coloring problem, we know that there exist graphs of class 2 with maximum degree at most five. As an extension for planar graphs with maximum degree six, we propose the following conjecture for $K_5$-minor free graphs.

\begin{conjecture} Let $G$ be a $K_5$-minor free graph with maximum degree $\Delta=6$. Then $G$ is class $1$.
\end{conjecture}


\end{document}